\documentclass[a4paper,10pt]{article}

\usepackage[latin1]{inputenc}
\usepackage[T1]{fontenc}
\usepackage[english]{babel}

\usepackage{mathrsfs}
\usepackage{amscd}
\usepackage{amsfonts}
\usepackage{amsmath}
\usepackage{amssymb}
\usepackage{amstext}
\usepackage{amsthm}
\usepackage{amsbsy}

\usepackage{xspace}
\usepackage[all]{xy}
\usepackage{graphicx}
\usepackage{url}
\usepackage{latexsym}

\usepackage{graphicx} 

\usepackage{booktabs} 
\usepackage{array} 
\usepackage{paralist} 
\usepackage{verbatim} 
\usepackage{subfig} 

\usepackage{fancyhdr} 
\usepackage{sectsty}
\allsectionsfont{\sffamily\mdseries\upshape} 

\usepackage[nottoc,notlof,notlot]{tocbibind} 
\usepackage[titles,subfigure]{tocloft} 

\usepackage[textwidth=100pt,textsize=footnotesize,bordercolor=white,color=blue!30]{todonotes}
\usepackage{hyperref} 

\makeatletter
\newcommand*{\rom}[1]{\expandafter\@slowromancap\romannumeral #1@}
\makeatother

\theoremstyle{definition}

\newtheorem{fact}{fact}

\newtheorem{thm}[fact]{Theorem}

\newtheorem{prop}[fact]{Proposition}
\newtheorem{corollary}[fact]{Corollary}
\newtheorem{defini}[fact]{Definition}

\title{A note on Autoreducibility for Infinite Time Register Machines and parameter-free Ordinal Turing Machines}
\author{Merlin Carl }
\date{\today} 

\begin{document}
\maketitle

\begin{abstract} We propose a notion of autoreducibility for infinite time computability and explore it and its connection with a notion of randomness for infinite time machines introduced
in \cite{CaSc} and \cite{Ca3}.
\end{abstract} 

\section{Autoreducibility for Infinite Time Register Machines}

The classical notion of autoreducibility can, for example, be found in \cite{DoHi}.
We consider how this concept behaves in the context of infinitary machine models of computations.
For the time being, we focus on Infinite Time Register Machines ($ITRM$s) (see \cite{ITRM} and \cite{ITRM2}) and ordinal Turing machines (see \cite{Ko}) - but the notion of course
makes sense for other types like the Infinite Time Turing Machines ($ITTM$s, see \cite{HaLe}) as well.

\begin{defini}
For $x\in ^{\omega}2$, we define $x_{\setminus n}$ as $x$ with its $n$th bit deleted (i.e. the bits up to $n$ are the same, the further bits are shifted one place to the left).
We say that $x$ is $ITRM$-autoreducible iff there is an $ITRM$-program $P$ such that $P^{x_{\setminus n}}(n)\downarrow=x(n)$ for all $n\in\omega$.
$x$ is called totally incompressible 
iff it is not $ITRM$-autoreducible, i.e. there is no $ITRM$-program $P$ such that $P^{x_{\setminus n}}(n)\downarrow=x(n)$ for all $n\in\omega$. If there is such a program,
then we say that $P$ autoreduces $x$, $P$ is an autoreduction for $x$ or that $x$ is autoreducible via $P$.
\end{defini}

\begin{defini}
 $x\in ^{\omega}2$ is $ITRM$-random in the measure sense iff there is no $ITRM$-decidable set $X$ of Lebesgue measure $0$ such that $x\in X$.
$x\in ^{\omega}2$ is $ITRM$-random in the meager sense iff there is no $ITRM$-decidable meager set $X$ such that $x\in X$.
\end{defini}

We refer the reader to \cite{CaSc} and \cite{Ca3} for more information on $ITRM$-randomness, including that used in the course of this note.

For the notion of $ITRM$-recognizability, we refer the reader to \cite{ITRM2}, \cite{Ca} or \cite{Ca2}.

\begin{corollary}
 No totally incompressible $x$ is $ITRM$-computable or even recognizable. $0^{\prime}_{ITRM}$, the real coding the halting problem for $ITRM$s, is $ITRM$-autoreducible.
\end{corollary}
\begin{proof}
 Clearly, if $P$ computes $x$, then $P$ is also an autoreduction  for $x$. If $x$ is recognizable and $P$ recognizes $x$, we can easily retrieve a 
deleted bit by pluggin in $0$ and $1$ and letting $P$ run on both results to see for which one $P$ stops with output $1$. (The same idea works for
finite subsets instead of single bits.) For $0^{\prime}_{ITRM}$, if a program index $j$ is given, it is easy to determine
some index $i\neq j$ corresponding to a program that works in exactly the same way (by e.g. adding a meaningless line somewhere), so that the remaining
bits allow us to reconstruct the $j$th bit. The autoreducibility of $0^{\prime}_{ITRM}$ also follows from the recognizability of $0^{\prime}_{ITRM}$ (see \cite{Ca2}).
\end{proof}

\begin{defini}
 Let $x\in^{\omega}2$, $i\in\omega$. Then $\text{flip}(x,i)$ denotes the real obtained from $x$ by just changing the $i$th bit, i.e. $x\Delta\{i\}$.
\end{defini}

In the classical setting, no random real is autoreducible. This is still true for $ITRM$s:

\begin{thm}{\label{randomnessimpliestotalincompressibility}}
 If $x$ is $ITRM$-random, then $x$ is totally incompressible. (For the meager as well as for the measure $0$ interpretation of randomness.)
\end{thm}
\begin{proof}
 Assume that $x$ is autoreducible via $P$. We show that $x$ is not $ITRM$-random. Let $X$ be the set of all $y$ which are autoreducible via $P$.
Obviously, we have $x\in X$. 
$X$ is certainly decidable: Given $y$, use a halting problem solver for $P$ to see whether $P^{y_{\setminus n}}(n)\downarrow$ for all $n\in\omega$.
If not, then $y\notin X$. Otherwise, carry out these $\omega$ many computations and check the results one after the other.\\
Since $X$ is $ITRM$-decidable, it is provably $\Delta_{2}^{1}$, which implies that $X$ has the Baire property 
and thus is measurable. 

We show that $X$ must be of measure $0$. To see this, assume for a contradiction that $\mu(X)>0$.
Note first, that, whenever $y$ is $P$-autoreducible and $z$ is a real that deviates from $y$ in exactly one digit (say, the $i$th bit), then $z$ is not $P$-autoreducible
(since $P$ will compute the $i$th bit wrongly).

By the Lebesgue density theorem, there is an open basic interval $I$ (i.e. consisting of all reals that start with a certain finite binary string $s$ length $k\in\omega$) such that
the relative measure of $X$ in $I$ is $>\frac{1}{2}$. Let $X^{\prime}=X\cap I$, and let $X^{\prime}_0$ and $X^{\prime}_1$ be the subsets of $X^{\prime}$ consisting 
of those elements that have their $(k+1)$th digit equal to $0$ or $1$, respectively. Clearly, $X^{\prime}_{0}$ and $X^{\prime}_{1}$ are measurable, $X^{\prime}_{0}\cap X^{\prime}_{1}=\emptyset$
and $X^{\prime}=X^{\prime}_{0}\cup X^{\prime}_{1}$. Now define $\bar{X}^{\prime}_{0}$ and $\bar{X}^{\prime}_{1}$ by changing the $(k+1)$th bit of all elements of $X^{\prime}_{0}$ and $X^{\prime}_{1}$,
respectively. Then all elements of $\bar{X}^{\prime}_{0}$ and $\bar{X}^{\prime}_{1}$ are elements of $I$ (as we have not changed the first $k$ bits), none of them is $P$-autoreducible (since they all
deviate from $P$-autoreducible elements by exactly one bit, namely the $k$th), $\bar{X}^{\prime}_{0}\cap\bar{X}^{\prime}_{1}=\emptyset$ (elements of the former set have $1$ as their $(k+1)$th digit,
for elements of $\bar{X}^{\prime}_{1}$ it is $0$) and $\mu(\bar{X}^{\prime}_{0})=\mu(X^{\prime}_{0})$, $\mu(\bar{X}^{\prime}_{1})=\mu(X^{\prime}_{1})$ (as the $\bar{X}^{\prime}_{i}$ are just 
translations of the $X^{\prime}_{i}$). As no element of the $\bar{X}^{\prime}_{i}$ is $P$-autoreducible, we have $(\bar{X}^{\prime}_{0}\cup\bar{X}^{\prime}_{1})\cap X^{\prime}=\emptyset$. Let
$\bar{X}^{\prime}:=\bar{X}^{\prime}_{0}\cup\bar{X}^{\prime}_{1}$.
Then we have\\
 $\mu_{I}(\bar{X}^{\prime})=\mu_{I}(\bar{X}^{\prime}_{0}\cup\bar{X}^{\prime}_{1})=\mu_{I}(\bar{X}^{\prime}_{0})+\mu_{I}(\bar{X}^{\prime}_{1})=\mu_{I}(X^{\prime}_{0})+\mu_{I}(X^{\prime}_{1})=\mu_{I}(X^{\prime})>\frac{1}{2}$ (where
$\mu_{I}$ denotes the relative measure for $I$). So $X^{\prime}$ and $\bar{X}^{\prime}$ are two disjoint subsets of $I$ both with relative measure $>\frac{1}{2}$, a contradiction.\\

For the meager version, we proceed similarly, taking $I$ to be an interval in which $X\cap I$ is comeager instead. That such an $I$ exists can be seen as follows:
Suppose that $X$ is not meager. As above, $X$ is $ITRM$-decidable, hence provably $\Delta_{2}^{1}$ and therefore has the Baire property. Then,
there is an open set $U$
such that $X\setminus U\cup U\setminus X$ is meager. In particular, $U$ is not empty. Hence $X$ is comeager in $U$. As $U$ is open, there is a nonempty open interval $I\subseteq U$.
It is now obvious that $X\cap I$ is comeager in $I$, so $I$ is as desired. We then use the same argument as above, noting that two comeager subsets of $I$ cannot be disjoint.
\end{proof}

\begin{corollary}{\label{Cohengenericsareincompressible}}
 Let $x$ be Cohen-generic over $L_{\omega_{\omega}^{CK}+1}$. Then $x$ is totally incompressible (in the measure sense).
\end{corollary}
\begin{proof}
Let $x$ be as in the assumption, and assume for a contradiction that $x$ is $P$-autoreducible. By the forcing theorem for 
provident sets (see \cite{Ma}), there must be a condition $p\subseteq x$ such that $p\Vdash\text{`}P\text{ autoreduces }\dot{x}\text{'}$, where $\dot{x}$ is a name for
the generic real (i.e. for $\bigcup\dot{G}$, where $\dot{G}$ is the canonical name for the generic filter). Let $i\in\omega\setminus\text{dom}(p)$ and let
$x^{\prime}:=\text{flip}(x,i)$. Then $x^{\prime}$ is still Cohen-generic over $L_{\omega_{\omega}^{CK}+1}$ and, as $p\subseteq x^{\prime}$, $p$ forces the 
$P$-autoreducibility of $x^{\prime}$; however, as $x^{\prime}$ differs from the $P$-autoreducible $x$ by only one bit, $x^{\prime}$ cannot be $P$-autoreducible,
a contradiction. Thus $x$ is not $P$-autoreducible.
\end{proof}

\begin{corollary}{\label{totallycompressiblesarerare}}
 The set of $ITRM$-autoreducible reals has measure $0$.
\end{corollary}
\begin{proof}
 The proof of Theorem \ref{randomnessimpliestotalincompressibility} shows that, for any $ITRM$-program $P$, the set of reals autoreducible via $P$ has measure $0$. As there are only countable
many programs, the result follows.
\end{proof}

\begin{defini}
 Denote by $IC_{ITRM}$ and $RA_{ITRM}$ the set of totally incompressible and $ITRM$-random reals, respectively (in the measure sense, for the time being).
\end{defini}

In this terminology, we showed above that $RA_{ITRM}\subseteq IC_{ITRM}$.
However, the converse of Theorem \ref{randomnessimpliestotalincompressibility} fails:

\begin{thm}{\label{totalincompressibilitydoesnotimplyrandomness}}
 $IC_{ITRM}\neq\subseteq RA_{ITRM}$, i.e. there is a real $x$ such that $x$ is totally incompressible, but not $ITRM$-random (in the measure sense).
\end{thm}
\begin{proof}
Let $X$ be an $ITRM$-decidable, comeager set of Lebesgue measure $0$. That such an $X$ exists is rather easy to see: 
A nice example for a comeager set of measure $0$ is the set of reals for which the zeros and ones in their
binary representation are not equally distributed. It is straightforward to implement a decision procedure for this set on an $ITRM$.
\\
Now, the set of Cohen-generic reals over $L_{\omega_{\omega}^{CK}+1}$ is comeager and hence must intersect $X$. Let $x\in X$ 
be Cohen-generic over $L_{\omega_{\omega}^{CK}+1}$. By Corollary \ref{Cohengenericsareincompressible}, $x$ is totally incompressible.
As $x\in X$ and $X$ is $ITRM$-decidable set of measure $0$, $x$ is not $ITRM$-random (in the measure sense).
Thus $x\in RA_{ITRM}\setminus IC_{ITRM}$, as desired. In fact, the set of these reals is comeager, as the set $C$ of Cohen-generic
reals over $L_{\omega_{\omega}^{CK}+1}$ is comeager, so that $C\cap X$ is also comeager and the proof shows that any element of $C\cap X$ is
of this kind.

\end{proof}

\subsection{Incompressibility and Randomness}

We saw above that the following inclusions hold (where $C^{+}$ denotes the set of Cohen-generic reals over $L_{\omega_{\omega}^{CK}+1})$:
\begin{center}
$C\subsetneq  RA_{ITRM}\subsetneq IC_{ITRM}$
\end{center}
(The first inclusion is proper because genericity for $\Pi_{1}$ and $\Sigma_1$-definable over $L_{\omega_{\omega}^{CK}}$ dense sets is sufficient, but not
every such real is generic over $L_{\omega_{\omega}^{CK}+1}$, which requires intersection with every definable dense set, $Pi_{1}/\Sigma_1$ or not. We do not know whether
$ITRM$-randomness can be characterized in terms of genericity in a natural way.)\\
In this section, we consider the question how similar incompressibility is to randomness, i.e. which of the results obtained for random reals also hold for incompressibles.\\

We start with an incompressible variant of the Kucera-Gacs theorem, which, as we recall, fails for $ITRM$-randomness, as no lost melody (an $ITRM$-recognizable real which is not
$ITRM$-computable; this was shown in \cite{Ca3}) is reducible to a random real.

\begin{thm}{\label{incompressiblekuceragacs}}
For every real $x$, there is a totally incompressible $y$ such that $x\leq_{ITRM} y$.
\end{thm}
\begin{proof}
 Given $x$, let $y$ be Cohen-generic over $L_{\omega_{\omega}^{CK,x}+1}[x]$ and let $z:=x\oplus y$. Then certainly $x\leq_{ITRM}z$.
Assume that $z$ is $P$-autoreducible for some program $P$. Hence, by the forcing theorem for provident sets \cite{Ma}, there is a condition $p$
such that $p\Vdash$`\v{x}$\oplus\bigcup\dot{G}$ is $P$-autoreducible', where $\dot{G}$ is the canonical name for the generic filter. 
The same hence holds for every $y^{\prime}$ which is Cohen-generic over $L_{\omega_{\omega}^{CK,x}+1}[x]$ with $p\subseteq y^{\prime}$.
Let $i\in\omega\setminus\text{dom}(p)$, $y^{\prime}:=\text{flip}(y,i)$, then $p$ forces the $P$-autoreducibility of $x\oplus y^{\prime}$. 
By absoluteness of computations, $x\oplus y^{\prime}$ is $P$-autoreducible. However, $x\oplus y^{\prime}$ differs from the $P$-autoreducible
$x\oplus y$ in exactly one bit and hence cannot be $P$-autoreducible, a contradiction.
\end{proof}

Also, in contrast to the theorem that $ITRM$-computability from mutually $ITRM$-random reals implies plain $ITRM$-computability, mutually incompressibles can contain
common non-trivial information ($COMP$ denotes the set of $ITRM$-computable reals):

\begin{defini}
$x$ is totally incompressible relative to $y$ ($y$-incompressible, incompressible in $y$) iff there is no program $P$ such that $P^{x_{\setminus n}\oplus y}\downarrow=x(n)$ for all $n\in\omega$.
If $x$ is $y$-incompressible and $y$ is $x$-incompressible, then $x$ and $y$ are mutually incompressible.
\end{defini}

\begin{thm}{\label{mutualincompressibility}}
 There are mutually incompressible reals $y,z$ and a real $x\notin COMP$ such that $x\leq_{ITRM}y$ and $x\leq_{ITRM}z$.
\end{thm}
\begin{proof}
 Let $y^{\prime}$, $z^{\prime}$ be mutually Cohen-generic over $L_{\omega_{\omega}^{CK,x}+1}[x]$, $y:=x\oplus y^{\prime}$, $z:=x\oplus z^{\prime}$ and apply the reasoning
of the proof of Theorem \ref{incompressiblekuceragacs}.

\end{proof}

\section{Ordinal Turing Machines}

See \cite{Ko} for an introduction to ordinal Turing machines.
For $OTM$s without parameters, define the notions of autoreducibility and total incompressibility as above for $ITRM$s. It turns out that there are no totally incompressible reals in $L$:\\

\begin{thm}{\label{noOTMincompressibles}}
Assume $V=L$. Then there are no totally $OTM$-incompressible reals.\end{thm}
\begin{proof}
Let $x\in L$. Our goal is to define a countable sequence $(P_{i}|i\in\omega)$ of programs deciding pairwise disjoint sets $(X_{i}|i\in\omega)$
with $\bigcup_{i\in\omega}X_{i}=\mathfrak{P}(\omega)$ such that if $y,z\in X$ differ only in finitely many bits, $y$ and $z$ are not in the same $X_{i}$.
Once that is done, the proof is easy to finish: There is some $i\in\omega$ such that $x\in X_i$, without loss of generality let $i=0$. Then $X_{0}$
is decided by $P_{0}$. Now an autoreduction for $x$ works as follows: Given $n\in\omega$ and $x_{\setminus n}$, plug $0$ and $1$ in for
the $i$th bit in $x_{\setminus n}$, getting reals $x_{0}$ and $x_{1}$, respectively, one of which is equal to $x$. Now use $P_{0}$ to decide whether $x_{0}\in X_{0}$
or $x_{1}\in X_{0}$. As $x_{0}$ and $x_{1}$ only differ in one bit and $X_{0}$ does not contain two (distinct) reals differing in only finitely many places,
only one of $x_{0}$ and $x_{1}$ can be an element of $X_{0}$, and that is $x$, determining the $n$th digit of $x$.\\
Now we construct $(P_{i}|i\in\omega)$  as follows: Let $(S_{i}|i\in\omega)$ be a natural enumeration of the finite sets of integers in order type $\omega$.
Write $y\sim z$ iff $y$ and $z$ differ only in finitely many bits. For a real $a$, denote by $[a]_{0}$ the $<_{L}$-smallest real such that $[a]_{0}\sim a$.
Then let $X_{i}:=\{[y]_{0}+_{b}S_{i}|y\in\mathfrak{P}^{L}(\omega\}$, where $+_{b}$ denotes the bitwise sum. Clearly, this is a countable partition of the constructible reals.
Furthermore, there is a decision procedure for $X_{i}$ on an $OTM$ (which is in fact uniform in $i$) which works as follows:
Given a real $a$ in the oracle, we can write $L$ on the tape until we arrive an $L$-level $L_{\alpha}\ni a$.
Then, searching $L_{\alpha}$, we can identify $[a]_{0}$. Now compute the set $S$ of bits where $a$ and $[a]_{0}$ differ and compare it to our enumeration of finite
subsets of $\omega$ fixed above: If $S=S_{i}$, then $a\in X_{i}$, otherwise $a\notin X_{i}$.
\end{proof}

Note that there are constructible reals which do not lie in any $OTM$-decidable null set, as the union $Y$ of all $OTM$-decidable null sets is an element in $L$
and, as a countable union of null sets, also a null set in $L$. Hence, at least in $L$, not every (parameter-free) $OTM$-random real is totally incompressible.\\

Note that the situation will probably be quite different for Infinite Time Turing Machines ($ITTM$s), as they have neither the power to enumerate $L$ nor the ability to solve their own restricted halting program (like $ITRM$s).\\

The $V=L$ hypothesis is probably unnecessarily strong here. However, even in rather mild extensions of $L$, $OTM$-incompressibles do exist:

\begin{thm}{\label{OTMincomprinCohenextension}}
 Let $x$ be Cohen-generic over $L$. Then $x$ is $OTM$-incompressible in $L[x]$ (and hence, by absoluteness of computations, in the real world).
\end{thm}
\begin{proof}
Assume for a contradiction that $x$ is $OTM$-autoreducible, say by the program $P$, where $x=\bigcup{G}$ and $G$ is a Cohen-generic
filter over $L$.
Then there is a finite $p\in G$ such that $p\Vdash\forall{n\in\omega}P^{x_{\setminus n}}(n)\downarrow=x(n)$.
Let $i\in \omega\setminus\text{dom}(p)$, $x^{\prime}:=\text{flip}(x,i)$. Then $x^{\prime}\in L[x]$ is still Cohen-generic over 
$L$ and $p\subseteq x^{\prime}$ so that $p\Vdash \forall{n\in\omega}P^{x^{\prime}_{\setminus n}}(n)\downarrow=x^{\prime}(n)$.
However, flipping a single bit cannot preserve $P$-autoreducibility, a contradiction. Hence $x$ is $OTM$-incompressible.
\end{proof}

Taking Theorem \ref{noOTMincompressibles} and Theorem \ref{OTMincomprinCohenextension} together, we get:

\begin{corollary}{\label{OTMincomprindependent}}
The existence of (parameter-free) $OTM$-incompressible reals is independent from $ZFC$.
\end{corollary}

Consequently, the analogue of Theorem \ref{randomnessimpliestotalincompressibility} for $OTM$s fails at least consistently: Every
constructible $OTM$-random real provides a counterexample.

For an $OTM$-program $P$, the set of $P$-autoreducibles is in general not decidable:

\begin{thm}{\label{OTMPcompressibilityundecidable}}
Assume that $V=L$. Then there are $OTM$-programs $P$ such that $X_{P}:=\{x\mid \forall{n\in\omega}P^{x_{\setminus n}}(n)\downarrow=x(n)\}$ (i.e. the set of $P$-autoreducibles) is not $OTM$-decidable.
\end{thm}
\begin{proof}
Assume for a contradiction that $X_{P}$ is decidable for every $P$. By the same argument as in the proof of Theorem \ref{randomnessimpliestotalincompressibility} then,
$\mu(X_{P})=0$ for every $P$. Consequently, no $OTM$-autoreducible real is $OTM$-random, and hence, every $OTM$-random real is $OTM$-incompressible.
However, the non-$OTM$-random reals are contained in a countable union of decidable null sets and hence form a null set themselves, so
that the $OTM$-random reals have full measure, while, on the other hand, $OTM$-incompressibles do not exist in $L$, a contradiction.
\end{proof}

Note, however, that $P$-autoreducibility for $OTM$s is semidecidable by simply simultaneously running all $OTM$-programs on a real $x$ and checking whether one
of them is an autoreduction. If such a program exists, it will eventually be found; otherwise, the search will not halt.\\

We note further that such sets are in general not measurable:

\begin{prop}
Assume $V=L$. Then there is an $OTM$-program $P$ such that the set of $P$-autoreducible reals is not measurable. In fact, there is a recursive set $I\subseteq\omega$
such that $\forall{x}P_{i}^{x}\downarrow=0\vee P_{i}^{x}\downarrow=1$, $S_{i}:=\{x\mid P_{i}^{x}\downarrow=1\}$ is not measurable, $S_{i}\cap S_{j}=\emptyset$ for $i\neq j$
 and $\mathfrak{P}(\omega)=\bigcup_{i\in\omega}S_{i}$.
\end{prop}
\begin{proof}
Let $(s_{i}|i\in\omega)$ be an enumeration of $^{<\omega}\omega$ in order type $\omega$, denote by $x\sim y$ that $x$ and $y$ differ only at finitely many places, let $[x]_{\sim}$
be the $\sim$-equivalence class of $x$
and let $P_{i}$ be the program described in the proof of Theorem \ref{noOTMincompressibles} that works as follows:
Given $x$ in the oracle, determine the $<_{L}$-minimal representative $x_{0}$ of $[x]_{\sim}$, then output $x_{0}\Delta s_{i}$ (where $\Delta$ denotes the symmetric difference, that is
we flip all the bits at places in $s_{i}$). Denoting, for $i\in\omega$, $E_{i}:=\{x\mid x=x_{0}\Delta s_{i}\}$, we have that $\mathfrak{P}(\omega)=\bigcup_{i\in\omega}E_{i}$,
$E_{i}\cap E_{j}=\emptyset$ for $i\neq j$ and $P_{i}$ decides $E_{i}$ for all $i,j\in\omega$. Furthermore, it is well-known that none of the $E_i$ is measurable.
\end{proof}

\section{Acknowledgements}
We are indebted to Philipp Schlicht for several very helpful discussions on the subject and in particular for the proof idea for Theorem \ref{totalincompressibilitydoesnotimplyrandomness}.

\end{document}